\documentclass[11pt,a4paper]{article}
\usepackage{amsmath,amssymb,latexsym,graphics,epsfig,tikz}
\usepackage{color}
\usepackage{amsthm}
\usepackage{graphicx,url}
\usepackage{hyperref}

\usepackage[english]{babel}
\usepackage{epsfig}
\usepackage{graphicx}

\newtheorem{thm}{Theorem}[section]

\newtheorem{lema}[thm]{Lemma}

\newtheorem{cor}[thm]{Corollary}
\newtheorem{ques}[thm]{Question}
\newtheorem{con}[thm]{Conjecture}
\newtheorem{rem}[thm]{Remark}

\title{Extending a conjecture of Graham and Lov\'{a}sz\\ on the distance characteristic polynomial}

\author{
Aida Abiad
\thanks{\texttt{a.abiad.monge@tue.nl}, Department of Mathematics and Computer Science, Eindhoven University of Technology, The Netherlands}
\thanks{Department of Mathematics: Analysis, Logic and Discrete Mathematics, Ghent University, Belgium} 
\thanks{Department of Mathematics and Data Science, Vrije Universiteit Brussel, Belgium} 
\and 
Boris Brimkov
\thanks{\texttt{boris.brimkov@sru.edu}, Department of Mathematics and Statistics, Slippery Rock University, U.S.A.} 
\and 
Sakander Hayat
\thanks{\texttt{sakander.hayat@giki.edu.pk}, Faculty of Engineering Sciences, GIK Institute of Engineering Sciences and Technology, Topi, Swabi, Pakistan} 
\and
Antonina P. Khramova
\thanks{\texttt{a.khramova@tue.nl}, Department of Mathematics and Computer Science, Eindhoven University of Technology, The Netherlands} 
\and
Jack H. Koolen
\thanks{\texttt{koolen@ustc.edu.cn}, School of Mathematical Sciences, University of Science and Technology of China, China} 
}

\date{}

\begin{document}

\maketitle

\begin{abstract}
Graham and Lov\'{a}sz conjectured in 1978 that the sequence of normalized coefficients of the distance characteristic polynomial of a tree of order $n$ is unimodal with the maximum value occurring at $\lfloor\frac{n}{2}\rfloor$. In this paper we investigate this problem for block graphs. In particular, we prove the unimodality part and we establish the peak for several extremal cases of uniform block graphs with small diameter.
\end{abstract}

\section{Introduction}\label{intro}
In their seminal work on spectral properties of the distance matrix $D$ of a tree $T$, Graham and Lov\'asz \cite{GL78} showed that, for a tree $T$ and $c_k(T)$ denoting the coefficient of $x^k$ in $\text{det}(D(T)-xI)=(-1)^np_{D(G)}(x)$, the quantity $d_k(T)=(-1)^{n-1}c_k(T)/2^{n-k-2}$  is determined as a fixed linear combination of the number of certain subtrees in $T$. The values $d_k(T)$ are called the \emph{normalized coefficients} of the distance characteristic polynomial of $T$. 

A sequence $a_0,a_1,a_2,\ldots, a_n$ of real numbers is {\em unimodal} if there is a $k$ such that $a_{i-1}\leq a_{i}$ for $i\leq k$ and $a_{i}\geq a_{i+1}$ for $i\geq k$. Graham and Lov\'asz \cite{GL78} conjectured that the sequence $d_0(T),\ldots,d_{n-2}(T)$ of normalized coefficients of the distance characteristic polynomial of a tree is unimodal with the maximum value occurring at $\lfloor\frac{n}{2}\rfloor$ for a tree $T$ of order $n$.
Little progress on this problem is known. Collins \cite{C89} confirmed the conjecture for stars, and also showed that for a paths the sequence is unimodal with a maximum value at $(1-\frac{1}{\sqrt{5}})n$. Thus, Graham and Lov\'asz conjecture was reformulated as follows:

\begin{con}\cite{GL78,C89}\label{con:coefficientstree}
The normalized coefficients of the distance characteristic polynomial of any tree with $n$ vertices are unimodal with peak between $\lfloor \frac{n}{2} \rfloor$ and $\lceil(1-\frac{1}{\sqrt{5}})n\rceil$.
\end{con}

Aalipour et al. \cite{GRWC2015} confirmed the unimodality part of Conjecture \ref{con:coefficientstree} by proving that the sequence of normalized coefficients is
indeed always unimodal; in fact, they proved the stronger statement that this sequence is log-concave. 


A natural and widely used generalization of trees are block graphs (also known as clique trees). A connected graph is a \emph{block graph} or \emph{clique tree} if its blocks ($2$-connected components) are cliques. Graham, Hoffman, and Hosoya \cite{GHH77} showed that the determinant of the distance matrix of any graph only depends on the determinant of its $2$-connected components.
Bapat and  Sivasubramanian \cite{BS} extended the result of the determinant of the distance matrix of a tree by Graham and Lov\'asz from \cite{GL78} to block graphs, and Das and Mohanty \cite{DM} did the same for multi-block graphs. Also the distance eigenvalues of a block graph have received some attention \cite{DLbook, LLL,JGZ21,XLS2020}. The study of distance matrix has a long history, and many results concerning the distance matrix and distance eigenvalues are reported in the literature, for a survey we refer the reader to the survey papers by Aouchiche and Hansen \cite{AH} and Hogben and Reinhart \cite{HR21}.

In this article we consider the sequence of coefficients $c_0,c_1,\dots,c_n$ of the distance characteristic polynomial of a block graph. Motivated by Conjecture \ref{con:coefficientstree}, we investigate the following question:

\begin{ques} \label{que:coefficientsblock}
 Is the sequence of coefficients $c_0,c_1,\dots,c_n$ of the distance characteristic polynomial of a block graph unimodal, and where does it peak?
\end{ques}

This paper is structured as follows. In Section \ref{preliminaries}, we start by recalling some definitions and preliminary results. In Section \ref{sec:emb} we develop a general theory regarding the eigenvalues of metric spaces. We use this in Section \ref{sec:unimodality} to show the unimodality part of Question \ref{que:coefficientsblock} for block graphs, which extends the corresponding result for trees by Aalipour {\it et al.}~\cite{GRWC2015}. In Section \ref{sec:peak} we prove that the peak part of Question~\ref{que:coefficientsblock} holds for several extremal classes of block graphs with small diameter. As a corollary of our results we obtain Collins' result for stars \cite{C89}. Although we show that the peak can move quite a bit when considering block graphs, our results provide evidence that Conjecture~\ref{con:coefficientstree} might still hold in a more general setting.

\section{Preliminaries}\label{preliminaries}

Throughout this paper, $G=(V,E)$ denotes an undirected, simple, connected and loopless graph with $n$ vertices. The \emph{distance matrix} $D(G)$ of a connected graph $G$ is the matrix indexed by the vertices of $G$ whose $(i,j)$-entry equals the distance  between the vertices $v_i$ and $v_j$, i.e., the length of a shortest path between $v_i$ and $v_j$. Dependence on $G$ may be omitted when it is clear from the context. The characteristic polynomial of $D$ is defined by $p_{D}(x)=\det(xI-D)$ and is called the \emph{distance characteristic polynomial} of $G$. Since $D$ is a real symmetric matrix, all of  the roots of the distance characteristic polynomial are real. 

A metric space can be attached to any connected graph $G =(V,E)$ in the following way. The \emph{path metric} of $G$, denoted $d_{G}$, is the metric where for all
vertices $u,v\in V$, $d_{G}(u,v)$ is the distance between $u$ and $v$ in $G$. Then, $(V,d_{G})$ is a metric space, called the \emph{graphic metric space}
associated with $G$.

A \emph{cut vertex} of a graph $G$ is a vertex whose deletion increases the number of connected components of $G$. A \emph{block} of $G$ is a maximal connected subgraph of $G$ which has no cut vertices. Thus, a block is either a maximal
2-connected subgraph, or a cut-edge, or an isolated vertex, and every such
subgraph is a block. Two blocks of $G$ can overlap in at most one vertex, which is a cut-vertex; hence, every edge of $G$ lies in a unique block, and $G$ is the
union of its blocks. A connected graph $G$  is called a \emph{block graph} if all of its blocks are cliques. In
particular, we say that $G$ is a \emph{$t$-uniform block graph} if all of its cliques have size $t$.

Given two graphs $G$ and $H$, their \emph{Cartesian product} is the graph $G\square H$ whose vertex
set is $V(G)\times V(H)$ and whose edges are the pairs $((a,x),(b,y))$ with $a,b\in V(G)$, $x,y\in V(H)$ and either $(a,b)\in E(G)$ and $x=y$, or $a=b$ and $(x,y)\in E(H)$.
The Cartesian product $H_{1}\square\cdots\square H_{k}$ of graphs $H_{1},\ldots,H_{k}$ is also denoted $\prod_{h=1}^{k}H_{h}$.

For integers $d\geq2$ and $n\geq2$, the \emph{Hamming graph} $H(d,n)$ is a graph whose vertex set is the $d$-tuples with elements from $\{0,1,2,\ldots,n-1\}$, where two vertices are adjacent if and only if their corresponding $d$-tuples differ only in one coordinate. Equivalently, $H(d,n)$ is the Cartesian product of $d$ copies of $K_n$. For a positive integer $d$, the \emph{hypercube} or the \emph{$d$-cube} is the graph $H(d,2)$.

A metric space $(X_1,d_1)$ is \emph{isometrically embeddable} in a metric space $(X_2,d_2)$ if there exists a mapping $\sigma:X_1\rightarrow X_2$ such that $d_2(\sigma(a), \sigma(b))=d_1(a,b)$. If $(X_1,d_1)$ and $(X_2,d_2)$ are graphic metric spaces associated with graphs $G$ and $H$ respectively, then $G$ is called an \emph{isometric subgraph} of $H$.
The cases when the graph $H$ is a hypercube, a Hamming graph, or
a Cartesian product of cliques of different sizes are of much importance. For other definitions and notations related to embeddability of metric spaces, we refer the reader to~\cite{TD1987, DLbook}.

As mentioned earlier, Graham, Hoffman, and Hosoya \cite{GHH77} proved the following elegant formula to calculate the determinant of the distance matrix of a block graph. Note that this formula depends only on the sizes of the blocks and not on the graph structure.
\begin{thm}[\cite{GHH77}]\label{GHH}
If $G$ is a block graph with blocks $G_{1},~G_2,~\ldots,~G_t$, then
\begin{eqnarray*}
\textnormal{cof}~D(G)&=&\prod\limits_{i=1}^{t}\textnormal{cof}~D(G_i),\\
\det D(G)&=&\sum\limits_{i=1}^{t}\det D(G_i)\prod\limits_{j\neq i}\textnormal{cof}~D(G_j).
\end{eqnarray*}
\end{thm}

The \emph{coefficient sequence} of a real polynomial $p(x)=a_nx_n+\cdots+a_1x+a_0$ is the sequence $a_0,a_1,a_2,\ldots,a_n$. The polynomial $p$ is \emph{real-rooted} if all roots of $p$ are real (by convention, constant polynomials are considered real-rooted). A sequence $a_0,a_1,a_2,\ldots, a_n$ of real numbers is {\em unimodal} if there is a $k$ such that $a_{i-1}\leq a_{i}$ for $i\leq k$ and $a_{i}\geq a_{i+1}$ for $i\geq k$, and the sequence is {\em log-concave} if $a_j^2\geq a_{j-1}a_{j+1}$ for all $j=1,\dots,  n-1$.

Finally, we will make use of the following well-known result (see for instance [\cite{2}, Lemma 1.1] and [\cite{5}, Theorem B, p. 270]) with the additional assumption that the polynomial coefficients are nonnegative, but it is straightforward to remove that assumption.

\begin{lema}\label{knownunimodalitylemma}
	\leavevmode
\begin{description}
\item[$(i)$] If $p(x)=a_nx_n+\cdots+a_1x+a_0$ is a real-rooted polynomial, then the coefficient sequence $a_i$ of $p$ is log-concave.
\item[$(ii)$] If the sequence $a_0,a_1,a_2,\ldots,a_n$ is positive and log-concave, then it is unimodal.
\end{description}
\end{lema}

\section{On $\ell_1$-embeddability of metric spaces}\label{sec:emb}


Metric spaces whose distance matrix has exactly one positive eigenvalue have received much attention since the work of Deza and Laurent~\cite{DLbook}. Metric properties of regular graphs have also been investigated by Koolen \cite{Koolenmscthesis}.

In this section we develop a general theory regarding the eigenvalues of metric spaces. The main result of this section, Theorem~\ref{thmfinal}, will be used to show the unimodality part of Question \ref{que:coefficientsblock}.

	
We say that a metric space $(X,d)$ is {\it of negative type} if for all weight functions $w:X\to\mathbb{Z}$ with $\sum\limits_{x\in X} w(x)=0$ we have 
\begin{equation}\label{neg-type-ineq}
\sum\limits_{x\in X}\sum\limits_{y\in X}w(x)w(y)d(x,y)\leq0.
\end{equation} 
If the same inequality holds for all weight functions with $\sum\limits_{x\in X} w(x)=1$, then we say that $(X,d)$ is {\it hypermetric}. It is fairly easy to show that $(X,d)$ being hypermetric implies it is of negative type. An alternative way to define metric spaces of negative type is by using Schoenberg's Theorem~\cite{Sch38}: a metric space $(X,d)$ is of negative type if and only if $(X,\sqrt{d})$ is isometrically embeddable in the Eucledian space. 
	
A metric space $(X,d)$ is said to be \textit{$\ell_1$-embeddable} if it can be embedded isometrically into the $\ell_{1}$-space ($\mathbb{R}^m, d_{\ell_{1}}$) for some integer $m\geq1$. Here, $d_{\ell_{1}}$ denotes the $\ell_{1}$-distance defined by
\begin{equation*}
d_{\ell_1}(x,y):=\sum\limits_{1\leq i\leq m}|x_i-y_i|~~\textrm{for}~
x,y\in \mathbb{R}^m.
\end{equation*}
One of the basic results of $\ell_{1}$-embeddable metric spaces is a characterization in terms of  cut semimetrics. Given a subset $S$ of the $n$-set $V_n:= \{1,\ldots,n\}$, the \emph{cut semimetric} $\delta_S$ is the distance on $V_n$ defined as
\begin{equation*}
\delta_S(i,j)=\left\{
\begin{array}{ll}
1, & \hbox{$i\in S,~j\in V_n\setminus S$,} \\
0, & \hbox{$i, j \in S$, \textrm{or} $i,j\in V_n\setminus S$.}
\end{array}
\right.
\end{equation*}
Note that every cut semimetric is clearly $\ell_{1}$-embeddable. In fact, a distance $d$ is $\ell_{1}$-\emph{embeddable} if and only if it can be decomposed as a nonnegative linear combination of cut semimetrics. We also note that $\ell_1$-embeddability of $(X,d)$ implies it is hypermetric~\cite{TD1987,Kelly67}.
	
	
Let $G$ be a graph and let $(X,d)$ be the graphic metric space and associated with~$G$. 
We have the following implications:
 
\begin{gather*}
(X,d) \text{ is $\ell_1$-embeddable} \\
		\Downarrow \\
(X,d) \text{ is hypermetric} \\
		\Downarrow \\
(X,d) \text{ is of negative type} \\
		\Downarrow \\
\text{The distance matrix of $G$ has exactly one positive eigenvalue}
\end{gather*}

\medskip

For more details on the metric hierarchy we refer the reader to the book by Deza and Laurent \cite{DLbook}.

Let ($X_1,d_1$) and ($X_2,d_2$) be two metric spaces. Their direct product is the metric space $(X_1\times X_2, d_1\otimes d_2)$ where, for $x_1,y_1\in X_1$, $x_2,y_2\in X_2$,
\begin{equation}\label{product-metric}
d_1\otimes d_2\big((x_1,x_2),(y_1,y_2)\big)=d_1(x_1,y_1)+d_2(x_2,y_2).
\end{equation}
Note that for path metrics, the direct product operation corresponds to the Cartesian product of graphs. Namely, if $G$ and $H$ are two connected graphs, then the direct product of their path metrics coincides with the path metric of the Cartesian product of $G$ and $H$. The following lemmas provide a relation between the metric hierarchy and the direct product of respective metric spaces.
\begin{lema}\label{l1-embed-spaces}
Let $d_i$ be a distance on the set $X_i$, for $i=1,2$. Then $(X_1\times X_2, d_1\otimes d_2)$ is $\ell_{1}$-embeddable if and only if both $(X_1,d_1)$ and $(X_2,d_2)$ are $\ell_{1}$-embeddable.
\end{lema}
\begin{proof}
		
Since both $d_1$ and $d_2$ are $\ell_{1}$-embeddable, they can be decomposed as a nonnegative linear combination of cut semimetrics. Therefore, if $d_1=\sum\limits_{S\subseteq X_1}a_{S}\delta_{S}$ and $d_2=\sum\limits_{T\subseteq X_2}b_{T}\delta_{T}$, then by (\ref{product-metric}) we obtain that
\begin{equation*}
d_1\otimes d_2=\sum\limits_{S\subseteq X_1}a_{S}\delta_{S\times X_2}+\sum\limits_{T\subseteq X_2}b_{T}\delta_{X_1\times T}.
\end{equation*}
This implies that $d_1\otimes d_2$ is decomposed into nonnegative linear combination of cut semimetrics. Thus $d_1\otimes d_2$ is $\ell_{1}$-embeddable.
\qedhere
\end{proof}
	
\begin{lema}\label{l2-embed-spaces}
Let $d_i$ be a distance on the set $X_i$, for $i=1,2$. Then $(X_1\times X_2, d_1\otimes d_2)$ is hypermetric (resp. of negative type) if and only if both ($X_1,d_1$) and ($X_2,d_2$) are hypermetric (resp. of negative type).
\end{lema}
\begin{proof}
Similar to Lemma~\ref{l1-embed-spaces}, it suffices to show that the inequality~(\ref{neg-type-ineq}) holds for $d_1\otimes d_2$ assuming $(X,d_1)$ and $(X,d_2)$ are hypermetric (resp. of negative type). By definition we have
\begin{align*}
&\sum\limits_{(x_1,x_2)\in X_1\times X_2}\sum\limits_{(y_1,y_2)\in X_1\times X_2} w(x_1,x_2)w(y_1,y_2)(d_1\otimes d_2)\left((x_1,x_2),(y_1,y_2)\right) = \\
&= \sum\limits_{x_1\in X_1}\sum\limits_{x_2\in X_2}\sum\limits_{y_1\in X_1}\sum\limits_{y_2\in X_2} w(x_1,x_2)w(y_1,y_2)\left(d_1(x_1,y_1)+d_2(x_2,y_2)\right) = \\
&= \sum\limits_{x_1\in X_1}\sum\limits_{y_1\in X_1}\left(\sum\limits_{x_2\in X_2}w(x_1,x_2)\right)\left(\sum\limits_{y_2\in X_2}w(y_1,y_2)\right)d_1(x_1,y_1)+ \\
&+ \sum\limits_{x_2\in X_2}\sum\limits_{y_2\in X_2}\left(\sum\limits_{x_1\in X_1}w(x_1,x_2)\right)\left(\sum\limits_{y_1\in X_1}w(y_1,y_2)\right)d_2(x_2,y_2).
\end{align*}
For a weight function $w:X_1\times X_2\to \mathbb{Z}$ such that 

$$\sum\limits_{(x_1,x_2)\in X_1\times X_2} w(x_1,x_2)=1\text{ (resp. $0$),}$$  we define $w_1:X_1\to\mathbb{Z}$ and $w_2:X_2\to\mathbb{Z}$ such that 
$$w_1(x)=\sum\limits_{x_2\in X_2} w(x,x_2)$$ and $$w_2(x)=\sum\limits_{x_1\in X_1} w(x_1,x).$$ Note that $$\sum\limits_{x_1\in X_1} w_1(x_1)=\sum\limits_{(x_1,x_2)\in X_1\times X_2}w(x_1,x_2)=1\text{ (resp. $0$),}$$ and thus for $(X_1,d_1)$ being hypermetric (resp. of negative type) and the weight function $w_1$ we have $$\sum\limits_{x_1\in X_1}\sum\limits_{y_1\in X_1} w_1(x_1) w_1(y_1)d_1(x_1,y_1)\leq 0.$$ Similarly, we have $$\sum\limits_{x_2\in X_2}\sum\limits_{y_2\in X_2} w_2(x_2)w_2(y_2)d_2(x_2,y_2)\leq 0.$$ The inequality~(\ref{neg-type-ineq}) for $d_1\otimes d_2$ then follows.
\end{proof}

\begin{thm}\label{thmfinal}
Let $G$ be the a graph whose $2$-connected components are of negative type. If $D(G)$ is the distance matrix of $G$, then $D(G)$ has exactly one positive eigenvalue.
\end{thm}
\begin{proof}
Lemma \ref{l2-embed-spaces} ensures us that the Cartesian product of the $2$-connected components of $G$ is of negative type. Since $G$ is an isometric subgraph of the Cartesian product, the result follows immediately.
\end{proof}

Terwilliger and Deza \cite{TD1987} investigated finite distance spaces having integral distances: a finite set $X$ and a map $d:X^2\rightarrow \mathbb{Z}$. The relation $d=1$ is assumed to be connected.  They provided the following classification of hypermetric spaces and metric spaces of negative type: $(X,d)$ has negative type if and only if it is metrically embeddable in a Euclidean space and generates a root lattice (direct sum of lattices of types $A,D,E$). Moreover, $(X,d)$ is hypermetric if and only if it is isomorphic to a subspace of the (complete) Cartesian products of the half-cubes, the CP-graphs, and the Johnson, Schläfli and Gosset graphs. These graphs correspond to the minimum vectors in the lattices dual to the root lattices mentioned above. Terwilliger and Deza conclude by describing how a given hypermetric $(X,d)$ may be embedded into a complete one. We should note that the results in \cite{TD1987} provide an alternative way to show Theorem~\ref{thmfinal} for hypermetric graphs. 


Next we observe that the Cartesian product of graphs does not preserve the one positive distance eigenvalue property. 

\begin{rem}
The Cartesian product of two graphs having one positive distance eigenvalues does not necessarily have one positive distance eigenvalue; see, e.g., Figure  \ref{fig:2evproduct}. In fact, if a vertex of $C_5$ is identified with any 4-degree vertex in the graph in Figure \ref{fig:2evproduct}, right, then the resulting 11-vertex graph has the following distance spectrum:
$$\{-7.83, -2.62, -2, -2, -1.38, -1, -1.0000, -0.38, -0.35, 0.08, 18.48\}.$$

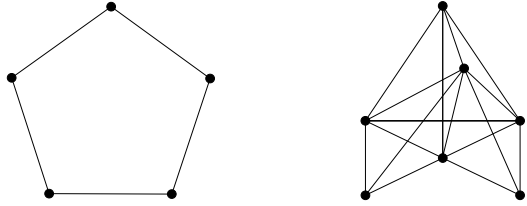
\begin{figure}[ht!]
\centering
\resizebox{7cm}{!}{
\begin{tikzpicture}[x=0.75pt,y=0.75pt,yscale=-1,xscale=1]

\draw   (121.78,141.11) -- (38.49,140.8) -- (13.05,61.5) -- (80.6,12.79) -- (147.8,61.99) -- cycle ;
\draw  [fill={rgb, 255:red, 0; green, 0; blue, 0 }  ,fill opacity=1 ] (35.53,140.8) .. controls (35.53,139.16) and (36.86,137.84) .. (38.49,137.84) .. controls (40.13,137.84) and (41.46,139.16) .. (41.46,140.8) .. controls (41.46,142.44) and (40.13,143.76) .. (38.49,143.76) .. controls (36.86,143.76) and (35.53,142.44) .. (35.53,140.8) -- cycle ;
\draw  [fill={rgb, 255:red, 0; green, 0; blue, 0 }  ,fill opacity=1 ] (118.81,141.11) .. controls (118.81,139.47) and (120.14,138.14) .. (121.78,138.14) .. controls (123.41,138.14) and (124.74,139.47) .. (124.74,141.11) .. controls (124.74,142.74) and (123.41,144.07) .. (121.78,144.07) .. controls (120.14,144.07) and (118.81,142.74) .. (118.81,141.11) -- cycle ;
\draw  [fill={rgb, 255:red, 0; green, 0; blue, 0 }  ,fill opacity=1 ] (10.08,61.5) .. controls (10.08,59.86) and (11.41,58.53) .. (13.05,58.53) .. controls (14.68,58.53) and (16.01,59.86) .. (16.01,61.5) .. controls (16.01,63.14) and (14.68,64.46) .. (13.05,64.46) .. controls (11.41,64.46) and (10.08,63.14) .. (10.08,61.5) -- cycle ;
\draw  [fill={rgb, 255:red, 0; green, 0; blue, 0 }  ,fill opacity=1 ] (144.84,61.99) .. controls (144.84,60.35) and (146.17,59.03) .. (147.8,59.03) .. controls (149.44,59.03) and (150.77,60.35) .. (150.77,61.99) .. controls (150.77,63.63) and (149.44,64.96) .. (147.8,64.96) .. controls (146.17,64.96) and (144.84,63.63) .. (144.84,61.99) -- cycle ;
\draw  [fill={rgb, 255:red, 0; green, 0; blue, 0 }  ,fill opacity=1 ] (77.64,12.79) .. controls (77.64,11.15) and (78.97,9.83) .. (80.6,9.83) .. controls (82.24,9.83) and (83.57,11.15) .. (83.57,12.79) .. controls (83.57,14.43) and (82.24,15.76) .. (80.6,15.76) .. controls (78.97,15.76) and (77.64,14.43) .. (77.64,12.79) -- cycle ;
\draw   (305.74,12.29) -- (358.27,90.93) -- (253.22,90.93) -- cycle ;
\draw   (358.27,141.93) -- (253.22,90.93) -- (358.27,90.93) -- cycle ;
\draw   (253.27,141.93) -- (358.27,90.93) -- (253.27,90.93) -- cycle ;
\draw    (305.74,12.29) -- (305.74,116.43) ;
\draw   (320.27,55.08) -- (305.74,115.58) -- (305.74,12.29) -- cycle ;
\draw    (320.27,55.08) -- (253.22,90.93) ;
\draw    (320.27,55.08) -- (253.27,141.93) ;
\draw    (358.27,90.93) -- (320.27,55.08) ;
\draw    (320.27,55.08) -- (358.27,141.93) ;
\draw  [fill={rgb, 255:red, 0; green, 0; blue, 0 }  ,fill opacity=1 ] (302.78,12.29) .. controls (302.78,10.65) and (304.1,9.33) .. (305.74,9.33) .. controls (307.38,9.33) and (308.71,10.65) .. (308.71,12.29) .. controls (308.71,13.93) and (307.38,15.26) .. (305.74,15.26) .. controls (304.1,15.26) and (302.78,13.93) .. (302.78,12.29) -- cycle ;
\draw  [fill={rgb, 255:red, 0; green, 0; blue, 0 }  ,fill opacity=1 ] (317.3,55.08) .. controls (317.3,53.44) and (318.63,52.11) .. (320.27,52.11) .. controls (321.9,52.11) and (323.23,53.44) .. (323.23,55.08) .. controls (323.23,56.71) and (321.9,58.04) .. (320.27,58.04) .. controls (318.63,58.04) and (317.3,56.71) .. (317.3,55.08) -- cycle ;
\draw  [fill={rgb, 255:red, 0; green, 0; blue, 0 }  ,fill opacity=1 ] (355.3,90.93) .. controls (355.3,89.3) and (356.63,87.97) .. (358.27,87.97) .. controls (359.9,87.97) and (361.23,89.3) .. (361.23,90.93) .. controls (361.23,92.57) and (359.9,93.9) .. (358.27,93.9) .. controls (356.63,93.9) and (355.3,92.57) .. (355.3,90.93) -- cycle ;
\draw  [fill={rgb, 255:red, 0; green, 0; blue, 0 }  ,fill opacity=1 ] (250.25,90.93) .. controls (250.25,89.3) and (251.58,87.97) .. (253.22,87.97) .. controls (254.85,87.97) and (256.18,89.3) .. (256.18,90.93) .. controls (256.18,92.57) and (254.85,93.9) .. (253.22,93.9) .. controls (251.58,93.9) and (250.25,92.57) .. (250.25,90.93) -- cycle ;
\draw  [fill={rgb, 255:red, 0; green, 0; blue, 0 }  ,fill opacity=1 ] (302.8,116.43) .. controls (302.8,114.8) and (304.13,113.47) .. (305.77,113.47) .. controls (307.4,113.47) and (308.73,114.8) .. (308.73,116.43) .. controls (308.73,118.07) and (307.4,119.4) .. (305.77,119.4) .. controls (304.13,119.4) and (302.8,118.07) .. (302.8,116.43) -- cycle ;
\draw  [fill={rgb, 255:red, 0; green, 0; blue, 0 }  ,fill opacity=1 ] (355.3,141.93) .. controls (355.3,140.3) and (356.63,138.97) .. (358.27,138.97) .. controls (359.9,138.97) and (361.23,140.3) .. (361.23,141.93) .. controls (361.23,143.57) and (359.9,144.9) .. (358.27,144.9) .. controls (356.63,144.9) and (355.3,143.57) .. (355.3,141.93) -- cycle ;
\draw  [fill={rgb, 255:red, 0; green, 0; blue, 0 }  ,fill opacity=1 ] (250.3,141.93) .. controls (250.3,140.3) and (251.63,138.97) .. (253.27,138.97) .. controls (254.9,138.97) and (256.23,140.3) .. (256.23,141.93) .. controls (256.23,143.57) and (254.9,144.9) .. (253.27,144.9) .. controls (251.63,144.9) and (250.3,143.57) .. (250.3,141.93) -- cycle ;

\end{tikzpicture}

}
\caption{Two graphs whose distance matrices have exactly one positive eigenvalue, but the graph resulting from joining them in one vertex has two positive distance eigenvalues.}
\label{fig:2evproduct}
\end{figure}
Moreover, from Zhang's and Godsil's result \cite[Theorem 3.3]{zg2013} it follows that if the distinguishing one vertex does not preserve the one distance eigenvalues property, then the Cartesian product does not preserve it either.
\end{rem}
 
\begin{rem} If  all the $2$-connected components of a graph $G$ have a full rank distance matrix, it does not necessarily follow that $D(G)$ has full rank. For example, consider the graph in Figure~\ref{fig:fullrank} with two biconnected components: $G'$ on vertices $\{0,1,\dots,7\}$ and $G''$ on vertices $\{0,8,9\}$. The distance matrices of $G$ and  $G'$ are:

{\tiny{
	\[
	D(G)=\left(
	\begin{array}{cccccccccc}
		0 & 1 & 2 & 1 & 2 & 2 & 2 & 2 & 1 & 1 \\
		1 & 0 & 1 & 1 & 2 & 2 & 2 & 2 & 2 & 2 \\
		2 & 1 & 0 & 1 & 1 & 1 & 1 & 1 & 3 & 3 \\
		1 & 1 & 1 & 0 & 1 & 1 & 1 & 1 & 2 & 2 \\
		2 & 2 & 1 & 1 & 0 & 1 & 2 & 2 & 3 & 3 \\
		2 & 2 & 1 & 1 & 1 & 0 & 2 & 2 & 3 & 3 \\
		2 & 2 & 1 & 1 & 2 & 2 & 0 & 2 & 3 & 3 \\
		2 & 2 & 1 & 1 & 2 & 2 & 2 & 0 & 3 & 3 \\
		1 & 2 & 3 & 2 & 3 & 3 & 3 & 3 & 0 & 1 \\
		1 & 2 & 3 & 2 & 3 & 3 & 3 & 3 & 1 & 0 \\
	\end{array}
	\right), \;\;\;
	D(G')=\left(
	\begin{array}{cccccccc}
		0 & 1 & 2 & 1 & 2 & 2 & 2 & 2 \\
		1 & 0 & 1 & 1 & 2 & 2 & 2 & 2 \\
		2 & 1 & 0 & 1 & 1 & 1 & 1 & 1 \\
		1 & 1 & 1 & 0 & 1 & 1 & 1 & 1 \\
		2 & 2 & 1 & 1 & 0 & 1 & 2 & 2 \\
		2 & 2 & 1 & 1 & 1 & 0 & 2 & 2 \\
		2 & 2 & 1 & 1 & 2 & 2 & 0 & 2 \\
		2 & 2 & 1 & 1 & 2 & 2 & 2 & 0 \\
	\end{array}
	\right),
	\]
	}}
and the ranks of $D(G')$ and $D(G'')$ are $8$ and $3$ respectively, so they are full rank matrices, whereas the rank of $D(G)$ is $9$, so it is not a full-rank matrix.
\end{rem}

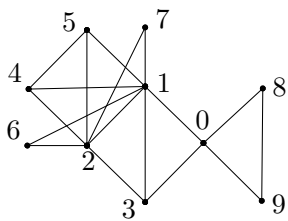
\begin{figure}[ht!]
	\centering
	\resizebox{4cm}{!}{
	\begin{tikzpicture}[x=0.75pt,y=0.75pt,yscale=-1,xscale=1]
		
		\draw    (141.87,52.67) -- (111.33,81) ;
		\draw [shift={(111.33,81)}, rotate = 137.14] [color={rgb, 255:red, 0; green, 0; blue, 0 }  ][fill={rgb, 255:red, 0; green, 0; blue, 0 }  ][line width=0.75]      (0, 0) circle [x radius= 1, y radius= 1]   ;
		\draw [shift={(141.87,52.67)}, rotate = 137.14] [color={rgb, 255:red, 0; green, 0; blue, 0 }  ][fill={rgb, 255:red, 0; green, 0; blue, 0 }  ][line width=0.75]      (0, 0) circle [x radius= 1, y radius= 1]   ;
		\draw    (111.33,81) -- (141.33,111) ;
		\draw [shift={(141.33,111)}, rotate = 45] [color={rgb, 255:red, 0; green, 0; blue, 0 }  ][fill={rgb, 255:red, 0; green, 0; blue, 0 }  ][line width=0.75]      (0, 0) circle [x radius= 1, y radius= 1]   ;
		\draw [shift={(111.33,81)}, rotate = 45] [color={rgb, 255:red, 0; green, 0; blue, 0 }  ][fill={rgb, 255:red, 0; green, 0; blue, 0 }  ][line width=0.75]      (0, 0) circle [x radius= 1, y radius= 1]   ;
		\draw    (141.87,52.67) -- (141.33,111) ;
		\draw [shift={(141.33,111)}, rotate = 90.52] [color={rgb, 255:red, 0; green, 0; blue, 0 }  ][fill={rgb, 255:red, 0; green, 0; blue, 0 }  ][line width=0.75]      (0, 0) circle [x radius= 1, y radius= 1]   ;
		\draw [shift={(141.87,52.67)}, rotate = 90.52] [color={rgb, 255:red, 0; green, 0; blue, 0 }  ][fill={rgb, 255:red, 0; green, 0; blue, 0 }  ][line width=0.75]      (0, 0) circle [x radius= 1, y radius= 1]   ;
		\draw    (111.33,81) -- (81.33,51.67) ;
		\draw [shift={(81.33,51.67)}, rotate = 224.36] [color={rgb, 255:red, 0; green, 0; blue, 0 }  ][fill={rgb, 255:red, 0; green, 0; blue, 0 }  ][line width=0.75]      (0, 0) circle [x radius= 1, y radius= 1]   ;
		\draw [shift={(111.33,81)}, rotate = 224.36] [color={rgb, 255:red, 0; green, 0; blue, 0 }  ][fill={rgb, 255:red, 0; green, 0; blue, 0 }  ][line width=0.75]      (0, 0) circle [x radius= 1, y radius= 1]   ;
		\draw    (81.33,111.67) -- (81.33,51.67) ;
		\draw [shift={(81.33,51.67)}, rotate = 270] [color={rgb, 255:red, 0; green, 0; blue, 0 }  ][fill={rgb, 255:red, 0; green, 0; blue, 0 }  ][line width=0.75]      (0, 0) circle [x radius= 1, y radius= 1]   ;
		\draw [shift={(81.33,111.67)}, rotate = 270] [color={rgb, 255:red, 0; green, 0; blue, 0 }  ][fill={rgb, 255:red, 0; green, 0; blue, 0 }  ][line width=0.75]      (0, 0) circle [x radius= 1, y radius= 1]   ;
		\draw    (111.33,81) -- (81.33,111.67) ;
		\draw [shift={(81.33,111.67)}, rotate = 134.37] [color={rgb, 255:red, 0; green, 0; blue, 0 }  ][fill={rgb, 255:red, 0; green, 0; blue, 0 }  ][line width=0.75]      (0, 0) circle [x radius= 1, y radius= 1]   ;
		\draw [shift={(111.33,81)}, rotate = 134.37] [color={rgb, 255:red, 0; green, 0; blue, 0 }  ][fill={rgb, 255:red, 0; green, 0; blue, 0 }  ][line width=0.75]      (0, 0) circle [x radius= 1, y radius= 1]   ;
		\draw    (81.33,51.67) -- (51.33,82.33) ;
		\draw [shift={(51.33,82.33)}, rotate = 134.37] [color={rgb, 255:red, 0; green, 0; blue, 0 }  ][fill={rgb, 255:red, 0; green, 0; blue, 0 }  ][line width=0.75]      (0, 0) circle [x radius= 1, y radius= 1]   ;
		\draw [shift={(81.33,51.67)}, rotate = 134.37] [color={rgb, 255:red, 0; green, 0; blue, 0 }  ][fill={rgb, 255:red, 0; green, 0; blue, 0 }  ][line width=0.75]      (0, 0) circle [x radius= 1, y radius= 1]   ;
		\draw    (81.33,111.67) -- (51.33,82.33) ;
		\draw [shift={(51.33,82.33)}, rotate = 224.36] [color={rgb, 255:red, 0; green, 0; blue, 0 }  ][fill={rgb, 255:red, 0; green, 0; blue, 0 }  ][line width=0.75]      (0, 0) circle [x radius= 1, y radius= 1]   ;
		\draw [shift={(81.33,111.67)}, rotate = 224.36] [color={rgb, 255:red, 0; green, 0; blue, 0 }  ][fill={rgb, 255:red, 0; green, 0; blue, 0 }  ][line width=0.75]      (0, 0) circle [x radius= 1, y radius= 1]   ;
		\draw    (81.33,51.67) -- (51.33,22.33) ;
		\draw [shift={(51.33,22.33)}, rotate = 224.36] [color={rgb, 255:red, 0; green, 0; blue, 0 }  ][fill={rgb, 255:red, 0; green, 0; blue, 0 }  ][line width=0.75]      (0, 0) circle [x radius= 1, y radius= 1]   ;
		\draw [shift={(81.33,51.67)}, rotate = 224.36] [color={rgb, 255:red, 0; green, 0; blue, 0 }  ][fill={rgb, 255:red, 0; green, 0; blue, 0 }  ][line width=0.75]      (0, 0) circle [x radius= 1, y radius= 1]   ;
		\draw    (51.33,82.33) -- (51.33,22.33) ;
		\draw [shift={(51.33,22.33)}, rotate = 270] [color={rgb, 255:red, 0; green, 0; blue, 0 }  ][fill={rgb, 255:red, 0; green, 0; blue, 0 }  ][line width=0.75]      (0, 0) circle [x radius= 1, y radius= 1]   ;
		\draw [shift={(51.33,82.33)}, rotate = 270] [color={rgb, 255:red, 0; green, 0; blue, 0 }  ][fill={rgb, 255:red, 0; green, 0; blue, 0 }  ][line width=0.75]      (0, 0) circle [x radius= 1, y radius= 1]   ;
		\draw    (51.33,22.33) -- (21.33,53) ;
		\draw [shift={(21.33,53)}, rotate = 134.37] [color={rgb, 255:red, 0; green, 0; blue, 0 }  ][fill={rgb, 255:red, 0; green, 0; blue, 0 }  ][line width=0.75]      (0, 0) circle [x radius= 1, y radius= 1]   ;
		\draw [shift={(51.33,22.33)}, rotate = 134.37] [color={rgb, 255:red, 0; green, 0; blue, 0 }  ][fill={rgb, 255:red, 0; green, 0; blue, 0 }  ][line width=0.75]      (0, 0) circle [x radius= 1, y radius= 1]   ;
		\draw    (51.33,82.33) -- (21.33,53) ;
		\draw [shift={(21.33,53)}, rotate = 224.36] [color={rgb, 255:red, 0; green, 0; blue, 0 }  ][fill={rgb, 255:red, 0; green, 0; blue, 0 }  ][line width=0.75]      (0, 0) circle [x radius= 1, y radius= 1]   ;
		\draw [shift={(51.33,82.33)}, rotate = 224.36] [color={rgb, 255:red, 0; green, 0; blue, 0 }  ][fill={rgb, 255:red, 0; green, 0; blue, 0 }  ][line width=0.75]      (0, 0) circle [x radius= 1, y radius= 1]   ;
		\draw    (81.33,51.67) -- (21.33,53) ;
		\draw [shift={(21.33,53)}, rotate = 178.73] [color={rgb, 255:red, 0; green, 0; blue, 0 }  ][fill={rgb, 255:red, 0; green, 0; blue, 0 }  ][line width=0.75]      (0, 0) circle [x radius= 1, y radius= 1]   ;
		\draw [shift={(81.33,51.67)}, rotate = 178.73] [color={rgb, 255:red, 0; green, 0; blue, 0 }  ][fill={rgb, 255:red, 0; green, 0; blue, 0 }  ][line width=0.75]      (0, 0) circle [x radius= 1, y radius= 1]   ;
		\draw    (20.67,82.33) -- (51.33,82.33) ;
		\draw [shift={(51.33,82.33)}, rotate = 0] [color={rgb, 255:red, 0; green, 0; blue, 0 }  ][fill={rgb, 255:red, 0; green, 0; blue, 0 }  ][line width=0.75]      (0, 0) circle [x radius= 1, y radius= 1]   ;
		\draw [shift={(20.67,82.33)}, rotate = 0] [color={rgb, 255:red, 0; green, 0; blue, 0 }  ][fill={rgb, 255:red, 0; green, 0; blue, 0 }  ][line width=0.75]      (0, 0) circle [x radius= 1, y radius= 1]   ;
		\draw    (20.67,82.33) -- (81.33,51.67) ;
		\draw [shift={(81.33,51.67)}, rotate = 333.18] [color={rgb, 255:red, 0; green, 0; blue, 0 }  ][fill={rgb, 255:red, 0; green, 0; blue, 0 }  ][line width=0.75]      (0, 0) circle [x radius= 1, y radius= 1]   ;
		\draw [shift={(20.67,82.33)}, rotate = 333.18] [color={rgb, 255:red, 0; green, 0; blue, 0 }  ][fill={rgb, 255:red, 0; green, 0; blue, 0 }  ][line width=0.75]      (0, 0) circle [x radius= 1, y radius= 1]   ;
		\draw    (81.33,21) -- (51.33,82.33) ;
		\draw [shift={(51.33,82.33)}, rotate = 116.06] [color={rgb, 255:red, 0; green, 0; blue, 0 }  ][fill={rgb, 255:red, 0; green, 0; blue, 0 }  ][line width=0.75]      (0, 0) circle [x radius= 1, y radius= 1]   ;
		\draw [shift={(81.33,21)}, rotate = 116.06] [color={rgb, 255:red, 0; green, 0; blue, 0 }  ][fill={rgb, 255:red, 0; green, 0; blue, 0 }  ][line width=0.75]      (0, 0) circle [x radius= 1, y radius= 1]   ;
		\draw    (81.33,21) -- (81.33,51.67) ;
		\draw [shift={(81.33,51.67)}, rotate = 90] [color={rgb, 255:red, 0; green, 0; blue, 0 }  ][fill={rgb, 255:red, 0; green, 0; blue, 0 }  ][line width=0.75]      (0, 0) circle [x radius= 1, y radius= 1]   ;
		\draw [shift={(81.33,21)}, rotate = 90] [color={rgb, 255:red, 0; green, 0; blue, 0 }  ][fill={rgb, 255:red, 0; green, 0; blue, 0 }  ][line width=0.75]      (0, 0) circle [x radius= 1, y radius= 1]   ;

		\draw (106,63.07) node [anchor=north west][inner sep=0.75pt] {$0$};
		\draw (86,43.73) node [anchor=north west][inner sep=0.75pt]  {$1$};
		\draw (47.33,84.4) node [anchor=north west][inner sep=0.75pt] {$2$};
		\draw (68,109.07) node [anchor=north west][inner sep=0.75pt]  {$3$};
		\draw (9.33,38.4) node [anchor=north west][inner sep=0.75pt]  {$4$};
		\draw (37.33,12.4) node [anchor=north west][inner sep=0.75pt]  {$5$};
		\draw (8.67,68.4) node [anchor=north west][inner sep=0.75pt]  {$6$};
		\draw (85.33,10.4) node [anchor=north west][inner sep=0.75pt]  {$7$};
		\draw (145.67,45.73) node [anchor=north west][inner sep=0.75pt]  {$8$};
		\draw (145.33,106.4) node [anchor=north west][inner sep=0.75pt]  {$9$};

	\end{tikzpicture}
	}
	\caption{A graph $G$ whose biconnected components have full rank distance matrices but its distance matrix $D(G)$ does not have full rank.}
	\label{fig:fullrank}
\end{figure}

\section{Extension of Graham and Lov\'{a}sz conjecture to block graphs}\label{sec:conjectureforblockgraphs}

Recall that, for a graph $G$, $c_k(G)$ denotes the coefficient of $x^k$ in $$\text{det}(D(G)-xI)=(-1)^np_{D(G)}(x).$$ The coefficients of the distance polynomial of a tree all have a common factor of $(-1)^{n-1}2^{n-k-2}$ due to a result of Graham and Lov\'{a}sz~\cite{GL78}. We say that the quantities $d_k(T)=(-1)^{n-1}c_k(T)/2^{n-k-2}$ are the \emph{normalized coefficients} of the distance characteristic polynomial of a tree $T$. Due to such common factor in  the coefficients of trees \cite{EGG76}, Conjecture~\ref{con:coefficientstree} uses the normalized coefficients. We should note that we will not use the normalized coefficients to investigate Question \ref{que:coefficientsblock}.

In this section we show that the sequence of coefficients of the distance characteristic polynomial of a block graph is unimodal, and we establish the peak for several extremal classes of block graphs with small diameter. 

\subsection{Unimodality}\label{sec:unimodality}

To answer the unimodality part of Question \ref{que:coefficientsblock}, we will follow a similar approach as it was done in \cite{GRWC2015} to show the unimodality of trees. The main idea relies on the fact that the distance matrix of a block graph on $n$ vertices has one positive and $n-1$ negative eigenvalues. We begin with a preliminary result, which extends the known result for trees by Edelberg, Garey and Graham \cite{EGG76}. 

\begin{lema}\label{claim}
The coefficients of the distance characteristic polynomial of a block graph $G$ satisfy
$$(-1)^{n-1}c_{k}(G)>0 \quad \text{for }0\leq k \leq n-2.$$ 
\end{lema}
\begin{proof}
It was shown in~\cite[Theorem~3.2]{LLL} that the distance matrix $D(G)$ of a block graph $G$ has one positive and $n-1$ negative eigenvalues.
We now extend  the argument given in ~\cite[Theorem 2.3]{EGG76} that $(-1)^{n-1}c_{k}(T)>0$ for $0\leq k\leq n-2$ for a tree $T$, given that its distance matrix has one positive  and $n-1$  negative eigenvalues. Let the eigenvalues of the distance matrix $D(G)$ of a block graph $G$ be denoted by $\lambda_1, -\lambda_2, \dots, -\lambda_n$, where $\lambda_i>0$ for~$i=1,\dots,n$. Then the distance characteristic polynomial is
\begin{align*}
	\det(D(G)-xI) &= (-1)^n(x-\lambda_1)(x+\lambda_2)\cdots(x+\lambda_n)= \\
	&= (-1)^n(x-\lambda_1)\sum\limits_{k=0}^{n-1} g_{n-1-k} x^{k}= \\
	&= (-1)^n\left(x^n+\sum\limits_{k=1}^{n-1}(g_{n-k}-\lambda_1g_{n-k-1})x^{k}-\lambda_1g_{n-1}\right),
\end{align*}
where $g_k$ is the sum of all $k$-fold products of $\lambda_2,\dots,\lambda_n$. Then, $c_{n-1}(G)=g_1-\lambda_1$, but also $c_{n-1}(G)=-c_n(G)\text{tr}(D)=0$; thus, $g_1=\lambda_1$. Then, since $g_{n-k}-\lambda_1g_{n-k-1}=g_{n-k}-g_1g_{n-k-1}<0$ for $k=1,\dots,n-2$, 
and since for $k=0$, $-\lambda_1g_{n-1}=-g_1g_{n-1}=-\prod\limits_{i=1}^{n}\lambda_i<0$, it follows that $(-1)^{n-1}c_{k}(G)>0$ for $0\leq k\leq n-2$.
\end{proof}



\begin{thm}\label{theo:unimodality}
For a block graph $G$, the sequence of coefficients of the distance characteristic polynomial $(-1)^{n-1}c_0(G),\ldots,(-1)^{n-1}c_{n-2}(G)$ is unimodal.
\end{thm}

\begin{proof}
First, it follows from Lemma \ref{claim} that if $G$ is a block graph, then the coefficients of the distance characteristic polynomial satisfy $(-1)^{n-1}c_{k}(G)>0 \quad \text{for }0\leq k \leq n-2$.

Since the distance matrix $D$ is a real symmetric matrix, the distance characteristic polynomial is real-rooted. From Lemma \ref{knownunimodalitylemma}$(i)$, it  follows that the sequence is log-concave.

Moreover, since
$$(-1)^{n-1}c_{k}(G)>0 \quad \text{for }0\leq k \leq n-2,$$
then Lemma \ref{knownunimodalitylemma}(ii) implies that the sequence of coefficients of the distance characteristic polynomial is unimodal.
\end{proof}

\subsection{Peak location}\label{sec:peak}

In this section we answer the peak location part of Question~\ref{que:coefficientsblock} for several extremal families of uniform block graphs with small diameter. 

The idea is to derive an explicit formula for the coefficients and use the unimodality to find the peak. However, while the method for obtaining the peak for stars and paths relies on the algebraic properties of the corresponding distance matrix \cite{C89}, for block graphs we will exploit several of its spectral properties.

Consider the \emph{windmill graph} $W(k,t)$, which is a block graph formed by joining $k$ cliques of size $t$ at a shared universal vertex. The following result uses Stirling's approximation to prove an estimate for the peak location of a windmill graph for large $k$. 

\begin{thm}\label{thm:friend.gen.approx} Consider a windmill graph $W(k,t)$ with $k\geq 2$ and $t\geq 3$, so that $n=|V|=k(t-1)+1$. Then, the sequence of coefficients is unimodal, and as $k$ approaches infinity the peak of the sequence occurs at $\frac{kt(t-1)}{2(t+1)}  + O(\log k)$.
\end{thm}


\begin{proof} The distance polynomial of the graph is
{\small{
$$p_D(x)=(x+1)^{(t-2)k}(x+t)^{k-1}(x^2-(t-2+2(t-1)(k-1))x-k(t-1)).$$
}}
To locate the peak of $p_D(x)$, it is sufficient to consider $(x+1)^{(t-2)k}(x+t)^{k-1}$.
	We know $(x+1)^{(t-2)k}$ peaks at $\frac{k(t-2)}2$. The coefficients of $(x+t)^{k-1}$ are defined by the formula $f_i = {k-1 \choose i}t^{k-1-i}$. By calculating explicitly $f_{\frac{k}{t+1}-2},f_{\frac{k}{t+1}-1},f_{\frac{k}{t+1}},f_{\frac{k}{t+1}+1}$ and using unimodality of binomial coefficients we conclude the sequence $f_i$ peaks at $\frac{k}{t+1}$.
	
	Define $$p = \frac{k(t-2)}2+\frac{k}{t+1}= \frac{kt(t-1)}{2(t+1)},$$ and let $d_i$ be the coefficients of $(x+1)^{(t-2)k}(x+t)^{k-1}$. Then  $$d_i=\sum\limits_{j=0}^i {(t-2)k \choose i-j} {k-1 \choose j} t^{k-1-j}.$$ We define $m_i = \max\limits_{j} {(t-2)k \choose i-j} {k-1 \choose j} t^{k-1-j}$, the maximal term in the sum. Then we have $d_i \leq k\cdot m_i$ for all $i$. On the other hand, $d_p$ has $${(t-2)k \choose (t-2)k/2} {k-1 \choose k/(t+1)} t^{k-1-k/(t+1)}$$ as one of the terms in its sum, so $d_p$ is greater than that. The idea is to show that $k m_i \leq {(t-2)k \choose (t-2)k/2} {k-1 \choose k/(t+1)} t^{k-1-k/(t+1)}$ for all $i\leq k-1$, which implies $d_i \leq d_p$.
	
	To find $m_i$, we apply Stirling's formula to ${(t-2)k \choose i-j} {k-1 \choose j} t^{k-1-j}$ and find its derivative with respect to $j$. Note that for $n\to\infty$, we have $$\log n! = n\log n-n+\frac12\log(2\pi n)+O(\log n),\text{ or }$$
$$n!-\left(\frac{n}{e}\right)^n\sqrt{2\pi}=n^{O(1)},$$
	meaning that $n!$ and $\frac{n^n}{e^n} \sqrt{2\pi n}$ are asymptotically equivalent. Also observe that since the sequence is unimodal it is sufficient to find a local peak among the coefficients with high values of $k$ and $n-k$. Hence, for large enough $n$, $k$, and $n-k$ we may derive
	{\small{
	$${n \choose k}= \frac{n!}{k!(n-k)!}\sim \frac{\sqrt{2\pi n} n^n e^{n-k} e^k}{2\pi \sqrt{k(n-k)} k^k (n-k)^{n-k} e^n} =\frac{n^n\sqrt{n}}{ \sqrt{2\pi k(n-k)} k^k (n-k)^{n-k}}.$$
	}}
	
	Then, we have
	{\small{\begin{align*}
	    &\left.{(t-2)k \choose i-j} {k-1 \choose j} t^{k-1-j}\right. \sim \\ 
	    &\sim \frac{((t-2)k)^{k(t-2)+1/2} (k-1)^{k-1+1/2}}{2\pi (i-j)^{i-j+1/2} ((t-2)k-i+j)^{(t-2)k-i+j+1/2} j^{j+1/2} (k-1-j)^{k-1-j+1/2}},
	\end{align*}}}
	and its derivative is then equal to $0$ if and only if
	{\small \begin{align*}
		\log\left(\frac{(i-j)(k-1-j)}{j (k(t-2)-i+j)}\right) &= \frac1{2j}-\frac1{2(i-j)}-\frac1{2(k-1-j)}+\frac1{2(k(t-2)-i+j)}, 
	\end{align*}}		
	{\small \begin{align*}		
		\frac{(i-j)(k-1-j)}{j (k(t-2)-i+j)} &= \sqrt{\frac{e^{1/j} e^{1/(k(t-2)-i+j)}}{e^{1/(i-j)} e^{1/(k-1-j)}}.} 
	\end{align*}}
Since $\sqrt{e^{1/z}}$ and $1$ are asymptotically equivalent for $z\to\infty$, we may assume
\begin{align*}
	\frac{(i-j)(k-1-j)}{j(k(t-2)-i+j)} &\sim 1, \\
	\frac{i(k-1)}{k(t-1)-1} &\sim j 
\end{align*}
For $k\to\infty$ we may assume $j= \frac{i}{t-1}$ and then the inequality we want to prove is

{\footnotesize{
$$k {(t-2)k \choose i-i/(t-1)} {k-1 \choose i/(t-1)} t^{k-1-i/(t-1)} \leq {(t-2)k \choose (t-2)k/2} {k-1 \choose k/(t+1)} t^{k-1-k/(t+1)}.$$
}}
This follows for the two inequalities:
{\small{
$$k {(t-2)k \choose i-i/(t-1)} t^{k/(t+1)} \leq {(t-2)k \choose (t-2)k/2} t^{i/(t-1)} \text{ and } {k-1 \choose i/(t-1)} \leq {k-1 \choose k/(t+1)}.$$
}}
The latter inequality follows from $\frac{i}{t-1}\leq\frac{k}{t+1}$ for $k\to\infty$. The former inequality can be shown using Stirling approximation.
\end{proof}

\begin{rem}
One may wonder on the extension to a windmill graph in which the cliques are not all of the same size. Consider a block graph with one universal vertex and $k_i$ cliques of size $t_i$ for $i\in\{1,\dots,l\}$ and some $l$, all sizes $t_1,\dots,t_l$ are distinct. Then the distance characteristic polynomial takes form
$$p_D(x) = (x+1)^{n-k-1} \prod\limits_{i=1}^l (x+t_i)^{k_i-1} p_B(x),$$
where $p_B(x)$ is the characteristic polynomial of the quotient matrix corresponding to the coarsest partition into $l+1$ vertex subsets $X_0,X_1,\dots X_l$: $X_0$ only contains the universal vertex, and $X_i$ has all vertices (except for the universal one) from all $k_i$ cliques of size $t_i$. Then the problem of finding $m_i$ as it is defined above is equivalent to a problem of finding a peak of an $l$-dimensional function, which complicates the application of the approximation approach even for small $l$. 
\end{rem}

The \emph{friendship graph} (or \emph{Dutch windmill graph}), denoted $F_{2k+1}$, is the graph $W(k,3)$ obtained by joining $k$ copies of $K_3$ by a common vertex so that $n=|V|=2k+1$. As a direct consequence of Theorem~\ref{thm:friend.gen.approx} we obtain the following corollary.

\begin{cor}
Let $F_{2k+1}$ be the friendship graph on $n=|V|=2k+1$ vertices. Then, the sequence of coefficients is unimodal, and as $k$ approaches infinity the peak of the sequence occurs at $3k/4 + O(\log k)$.
\end{cor}

Using a different and non-asymptotic approach, next we show that the windmill graph with 2 cliques, $W(2,t)$, has distance characteristic polynomial coefficients with peak exactly at $\lfloor\frac{n}{2}\rfloor$. 

\begin{thm}
Let $G$ be a graph obtained by adding a universal vertex to the disjoint union of two cliques $K_t$, and let $n=|V(G)|=2t+1$. Then the sequence of coefficients of the distance characteristic polynomial of $G$ is unimodal with peak at $\lfloor\frac{n}{2}\rfloor$.
\end{thm}
\begin{proof}
Consider an equitable partition of $G$ into three subsets, one is the universal vertex and the other two correspond to cliques $K_t$. Using the quotient matrix of this partition, we can compute the distance characteristic polynomial of $G$ to be
\begin{eqnarray*}
p_{D}(x)&=&(x+1)^{2t-2}(-(t+1)-x)(-2t-(3t-1)x+x^2) =\\
&=&(x+1)^{2t-2}(ax^3 + bx^2 + cx + d),
\end{eqnarray*}
where $a=-1$, $b=2t-2$, $c=3t^2+4t-1$, and $d=2t^2+2t$.
Multiplying the binomial expansion of $(x+1)^{2t-2}$ by $(ax^3 + bx^2 + cx + d)$ and combining the coefficients of terms with the same power, we obtain
{\allowdisplaybreaks{
{\small{
\begin{align*}
p_{D}(x)&=\left[d{2t-2\choose 0}\right]x^0 +\\
&+\left[c{2t-2\choose 0}+d{2t-2\choose 1}\right]x^1+\\
&+\left[b{2t-2\choose0}+c{2t-2\choose 1}+d{2t-2\choose 2}\right]x^2+\\
&+\sum_{i=0}^{2t-5}\left[a{2t-2\choose i} +b{2t-2\choose i+1}+c{2t-2\choose i+2}+d{2t-2\choose i+3}\right]x^{i+3}+\\
&+\left[a{2t-2\choose 2t-4} + b{2t-2\choose 2t-3} + c{2t-2\choose 2t-2}\right]x^{2t-1}+\\
&+\left[a{2t-2\choose 2t-3}+b{2t-2\choose 2t-2}\right]x^{2t}+\\
&+\left[a{2t-2 \choose 2t-2}\right]x^{2t+1}.
\end{align*}
}}
}}
Let $c_i$ be the coefficient of $x^i$ in $p_{D}(x)$. Then, for $t\geq 4$ and $j\geq 0$,
\begin{equation}
\label{beq1}
c_{t-j}=a{2t-2\choose t-3-j} +b{2t-2\choose t-2-j}+c{2t-2\choose t-1-j}+d{2t-2\choose t-j}.
\end{equation}
Using the formula ${n\choose k} = \frac{(n - k + 1)}{k} { n\choose k - 1}$, we can rewrite \eqref{beq1} as

\begin{align*}
c_{t-j} &= \binom{2 t-2}{t-3-j}
\left(
a
+b\cdot\frac{j+t+1}{t-j-2}
+c\cdot\frac{j+t+1}{t-j-2}\cdot\frac{j+t}{t-j-1}\right. +\\
 &+\left. d\cdot\frac{j+t-1}{t-j-2}\cdot\frac{j+t}{t-j-1}\cdot\frac{j+t+1}{t-j}\right) = \\
&= \binom{2 t-2}{t-3-j} \cdot f(t,j),
\end{align*}
where $f(t,j)=\frac{t (j+t) \left(t \left(j^2+j (3-4 t)-t (5 t+11)+2\right)+2\right)}{(j-t) (j-t+1) (j-t+2)}$.
Then, $c_{t-j}\geq c_{t-(j+1)}$ if and only if
$$\binom{2 t-2}{t-3-j} \cdot f(t,j)\geq \binom{2 t-2}{t-3-(j+1)} \cdot f(t,j+1),$$
which is equivalent to
\begin{equation}
\label{beq2}
\frac{f(t,j)}{t-3-j}\geq \frac{f(t,j+1)}{t+j+2}.
\end{equation}
It can be verified using a software with symbolic algebra for rational functions that \eqref{beq2} holds for all integers $j$ and $t$ with $t\geq 4$ and $0\leq j < t-3$; for $t\leq 3$, it can be verified that $c_{t-j}\geq c_{t-(j+1)}$ by explicitly computing the distance characteristic polynomial. Similarly as above, it can be shown that $c_{t+j}\geq c_{t+(j+1)}$ for all $t$ and $j\geq 0$. Thus, it follows that $c_0\leq \ldots\leq c_t\geq \ldots \geq c_{2t+1}$, and hence the distance characteristic polynomial of $G$ is unimodal with peak  at $t = \lfloor \frac{2t+1}{2}\rfloor=\lfloor \frac{n}{2}\rfloor$.
\end{proof}

Consider now the class of \textit{barbell graphs} $B(t,\ell)$, which are obtained by connecting two cliques $K_t$ with a path on $\ell$ vertices by identifying the leaves of the path with one of the vertices in each clique. We also consider \textit{lollipop graphs} $L(t,\ell)$, which are obtained by adding a path on $\ell$ vertices to a single clique $K_t$ so that one of the leaves of the path is also a vertex of the clique. We begin by locating the peak of the distance  characteristic polynomial for a barbell graph with $\ell\in\{2,3,4,5\}$.

\begin{thm}\label{thm.barbell}	
	Let $G$ be a barbell graph $B(t,\ell)$, so that $|V(G)|=n=2t+\ell-2$. Then the sequence of coefficients of the distance characteristic polynomial of $G$ is unimodal with peak at $t-1$ if $\ell=2$ and at $t$ if $\ell\in\{3,4,5\}$.
\end{thm}

\begin{proof}
We prove the claim for $\ell=2$; the cases $\ell\in\{3,4,5\}$ can be shown analogously.
Consider a quotient partition into $4$ vertex sets: $2$ of them are the two vertices of the path of length $2$ connecting the cliques and the other $2$ correspond to the remaining $t-1$ vertices of each clique. We then have the quotient matrix
$$B=\left(
	\begin{array}{cccc}
		t-2 & 1 & 2 & 3 (t-1) \\
		t-1 & 0 & 1 & 2 (t-1) \\
		2 (t-1) & 1 & 0 & t-1 \\
		3 (t-1) & 2 & 1 & t-2 \\
	\end{array}
	\right)$$
	with characteristic polynomial $$p_B(x)=-x^4+(2 t-4) x^3+\left(8 t^2-4 t-4\right) x^2+\left(14 t^2-12t\right) x+5 t^2-4 t.$$ 
	From \cite[Theorem 3.3]{HR21}, it follows that the distance characteristic polynomial of $G$ is 
	{\footnotesize{
	$$p_D(x)=(x+1)^{2t-4}\left(-x^4+(2 t-4) x^3+\left(8 t^2-4 t-4\right) x^2+\left(14 t^2-12t\right) x+5 t^2-4 t\right).$$
	}}
	
Let $a=-1$, $b=2t-4$, $c=8t^2-4t-4$, $d=14t^2-12t$, and $e=5t^2-4t$. We can use $$(x+1)^{2t-4}=\sum\limits_{k=0}^{2t-4}\binom{2t-4}{k} x^k$$ to write down a formula in the case $5\leq k\leq 2t-4$:
	$$c_{k}=a\binom{2t-4}{k-4}+b\binom{2t-4}{k-3}+c\binom{2t-4}{k-2}+d\binom{2t-4}{k-1}+e\binom{2t-4}{k}.$$
	Using the identity $\binom{n}{k}=\frac{n-k+1}k\binom{n}{k-1}$ we obtain

 {\small{
\begin{align*}
	c_{k}&=\binom{2t-4}{k-4}\left(a+\frac{2t-k}{k-3}\cdot b+\frac{2t-k-1}{k-2}\cdot\frac{2t-k}{k-3}\cdot c+\right. \\
	&+\left.\frac{2t-k-2}{k-1}\cdot\frac{2t-k-1}{k-2}\cdot\frac{2t-k}{k-3}\cdot d\right. + \\ &+\left.\frac{2t-k-3}{k}\cdot\frac{2t-k-2}{k-1}\cdot\frac{2t-k-1}{k-2}\cdot\frac{2t-k}{k-3}\cdot e \right) = \\
	&= \binom{2t-4}{k-4}f(t,k).
\end{align*}
 }}
	Then, $c_k\geq c_{k-1}$ if and only if
	$$\binom{2t-4}{k-4}f(t,k)\geq \binom{2t-4}{k-5}f(t,k-1),$$
	which by applying the identity $\binom{n}{k}=\frac{n-k+1}k\binom{n}{k-1}$ leads to
	$$\frac{f(t,k)}{k-4}\geq \frac{f(t,k-1)}{2t-k+1}.$$
	If $k=t$ then the simplified form of the above inequality is
	$$\frac{21 t^4-35 t^3+t^2+3 t+6}{t^4-8 t^3+17 t^2+2 t-24}\leq 0,$$
	and for $k=t-1$ we have $$\frac{33 t^5+41 t^4-183 t^3+7 t^2+114 t-24}{(t-5) (t-4) (t-3)
		(t-2) (t+2)}\geq 0.$$
	It is straightforward to verify that if $t\geq 6$, the inequality for the $k=t-1$ case holds, but the one for $k=t$ does not, meaning $c_{t-2}\leq c_{t-1}> c_t$. Since the sequence of coefficients is unimodal by Theorem~\ref{theo:unimodality}, this implies that $t-1$ is indeed the peak location. The cases $t\leq5$ can be checked by straightforward calculation.
\end{proof}

An analogous argument as in the proof of Theorem \ref{thm.barbell}	 can also be used to show the peak location of lollipop graphs with small~$\ell$.

\begin{cor} Let $G$ be a lollipop graph $L(t,\ell)$, so that $|V(G)|=n=t+\ell-1$. Let $\ell\in\{2,3,4,5\}$. Then the sequence of coefficients of the distance characteristic polynomial of $G$ is unimodal with peak at $\lfloor\frac{n-1}2\rfloor$.
\end{cor}

\section{Concluding remarks}

We end up by discussing another extremal case of block graphs, namely, block paths. According to SageMath simulations, the peak of a block path seems to be located between $\lfloor\frac{n}{3}\rfloor$ and $\lfloor\frac{n}{2}\rfloor$. A natural approach to extend the result from paths to block paths would be to generalize Collins proof for the distance characteristic polynomial of paths~\cite{C89} and calculate the distance characteristic polynomial coefficients of block paths using the formula for edge-weighted block graphs (see Theorem~\ref{GHH}). In general, the coefficient formula is given by
$$c_{n-k}=(-1)^{n-k}\sum\limits_{1\leq i_1\leq\dots\leq i_k\leq n}\det D[v_{i_1},\dots,v_{i_k}],$$
where $D[v_{i_1},\dots,v_{i_k}]$ is a $k\times k$ submatrix of a $D$ whose rows and columns are indexed by vertices $v_{i_1},\dots,v_{i_k}$. The sum is over all possible ways to choose $k$ out of $n$ vertices.
The main idea of Collins' approach is to interpret $D[v_{i_1},\dots,v_{i_k}]$ as a distance matrix of an edge-weighted block graph, and then use Theorem~\ref{GHH} to calculate its determinant. However, for the case when the blocks within the path have size at least $3$, the description of blocks of $D[v_{i_1},\dots,v_{i_k}]$ can get very irregular: depending on a particular choice of $k$ vertices, they can be anything from $2$- to $k$-cliques with weights of edges seemingly following no pattern. Thus, for the more general setting of block graphs, it seems hopeless to use an analogous approach as Collins does for trees.

Our computational results suggest the following question:

\begin{ques}
Are the coefficients of the distance characteristic polynomial of any block graph with $n$ vertices unimodal with peak between $\lfloor\frac{n}{3}\rfloor$ and $\lfloor\frac{n}{2}\rfloor$?
\end{ques}

 We conclude by observing that the formula for the normalized coefficients of the characteristic polynomial of a tree \cite[page 81]{GL78} probably also holds for uniform block graphs. If this was the case, then we could define the normalized coefficients for a $t$-uniform block graph (the distance eigenvalues of Hamming graphs \cite[Section 3.1]{GRWC2015v2} suggest that there exists such a common factor).

\subsection*{Acknowledgements}

Aida Abiad is partially supported by the FWO (Research Foundation Flanders), grant number 1285921N. Antonina P. Khramova is supported by the NWO (Dutch Science Foundation), grant number OCENW.KLEIN.475. Jack H.~Koolen is partially supported by the National Natural Science Foundation of China (No. 12071454), Anhui Initiative in Quantum Information Technologies (No. AHY150000) and the National Key R and D Program of China (No. 2020YFA0713100).


\end{document}